\documentclass[12pt]{amsart}
\usepackage{graphicx}
\usepackage[headings]{fullpage}
\usepackage{amssymb,epic,eepic,epsfig,amsbsy,amsmath,amscd}
\numberwithin{equation}{section}
                        \theoremstyle{plain}
\usepackage{mathrsfs}

\newcommand\no[1]{}

\newtheorem{theorem}{Theorem}[section]
\newtheorem{thm}{Theorem}
\newtheorem{lemma}[theorem]{Lemma}

\newtheorem{proposition}[theorem]{Proposition}

\theoremstyle{definition}
\newtheorem{remark}[theorem]{Remark}

\def\BC{\mathbb C}

\def\BZ{\mathbb Z}
\def\BR{\mathbb R}
\def\BT{\mathbb T}

\def\fb{\mathfrak b}

\def\la{\langle}
\def\ra{\rangle}

\DeclareMathOperator{\tr}{\mathrm tr}

\def\ve{\varepsilon}
\def\be { \begin{equation} }
\def\ee { \end{equation} }

\begin{document}

\title[Volumes of double twist knot cone-manifolds]
{Volumes of double twist knot cone-manifolds}

\author{Anh T. Tran}
\address{Department of Mathematical Sciences, The University of Texas at Dallas, Richardson, TX 75080, USA}
\email{att140830@utdallas.edu}


\begin{abstract}
We give explicit formulae for the volumes of hyperbolic cone-manifolds of double twist knots, a class of two-bridge knots which includes twist knots and two-bridge knots with Conway notation $C(2n,3)$. We also study the Riley polynomial of a class of one-relator groups which includes two-bridge knot groups.
\end{abstract}

\maketitle

\section{Introduction}

Let $K$ be a hyperbolic knot and $X_K$ the complement of $K$ in $S^3$. Let $\rho_{\text{hol}}$ be a holonomy representation of $\pi_1(X_K)$ into $PSL_2(\BC)$. Thurston \cite{Th} showed that $\rho_{\text{hol}}$ can be deformed into a one-parameter family $\{\rho_{\alpha}\}_\alpha$ of representations to give a corresponding one-parameter family $\{C_{\alpha}\}_\alpha$ of singular complete hyperbolic manifolds. These $\alpha$'s and $C_{\alpha}$'s are called the cone-angles and  hyperbolic cone-manifolds of $K$, respectively. 

We consider the complete hyperbolic structure on a knot complement as the cone-manifold structure of cone-angle zero. It is known that for a two-bridge knot $K$ there exists an angle $\alpha_K \in [\frac{2\pi}{3},\pi)$ such that $C_\alpha$ is hyperbolic for $\alpha \in (0, \alpha_K)$, Euclidean for $\alpha =\alpha_K$, and spherical for $\alpha \in (\alpha_K, \pi)$ \cite{HLM, Ko, Po, PW}. In \cite{HLM} a method for calculating the volumes of two-bridge knot cone-manifolds were introduced but without explicit formulae. Explicit volume formulae for hyperbolic cone-manifolds of knots are known for the following knots: $4_1$ \cite{HLM, Ko, Ko2, MR}, $5_2$ \cite{Me},  twist knots \cite{HMP} and two-bridge knots with Conway notation $C(2n,3)$ \cite{HL}. In this paper we will calculate the volumes of cone-manifolds of double twist knots, a class of two-bridge knots which includes twist knots and two-bridge knots with Conway notation $C(2n,3)$.

Let $J(k, l)$ be the knot/link as in Figure 1, where $k,l$ denote 
the numbers of half twists in the boxes. Positive (resp. negative) numbers correspond 
to right-handed (resp. left-handed) twists. Note that $J(k, l)$ is a knot if and only if $kl$
is even, and is the trivial knot if $kl = 0$. Furthermore, $J(k, l) \cong J(l, k)$ and $J(-k, -l)$ is
the mirror image of $J(k, l)$. Hence, without loss of generality, we will only consider $J(k, 2n)$
for $k > 0$ and $|n| > 0$. The knot $J(2,2n)$ is known as a twist knot, and $J(3,2n)$ is the two-bridge knot with Conway notation $C(-2n,3)$. In general, the knot $J(k,2n)$ is called a double twist knot. It is a hyperbolic knot if and only if $|k|, |2n| \ge 2$ and $J(k,2n)$ is not the trefoil knot. We will now exclusively consider the hyperbolic $J(k,2n)$ knots.

\begin{figure}[th]
\centerline{\psfig{file=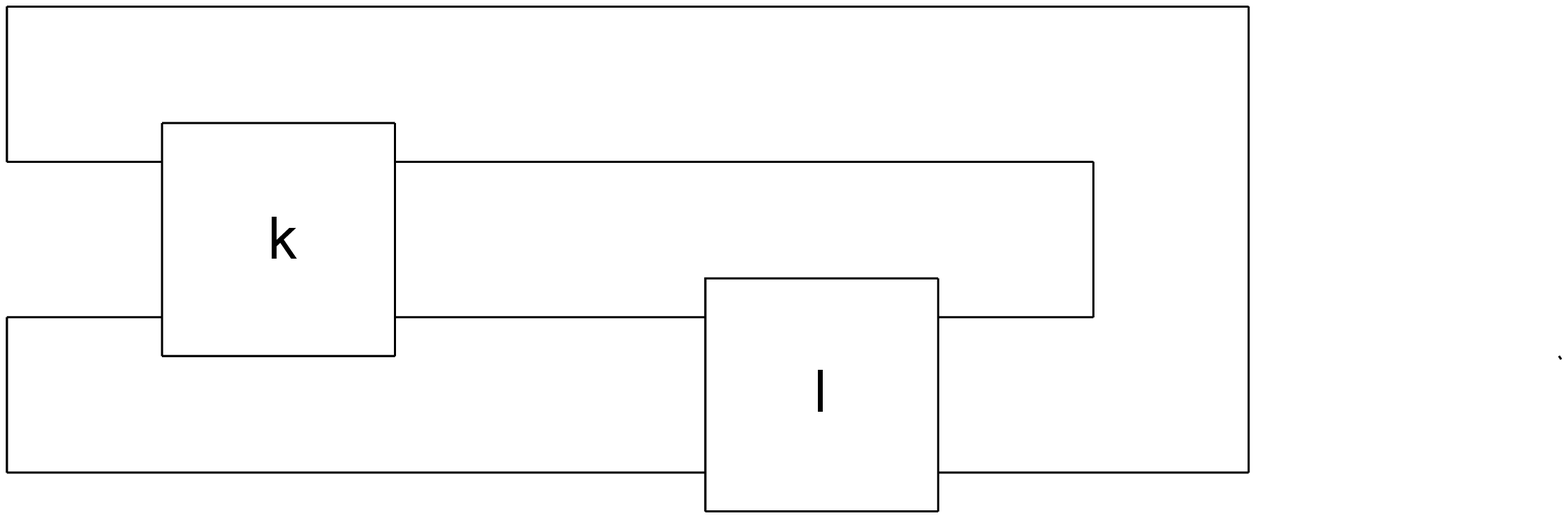,width=3.5in}}
\vspace*{8pt}
\caption{The knot/link $J(k,l)$. }
\end{figure} 

To state our main results we introduce the Chebychev polynomials of the second kind $S_j(\omega)$. They are recursively defined by $S_0(\omega)=1$, $S_1(\omega)=\omega$ and $S_{j}(\omega) = \omega S_{j-1}(\omega) - S_{j-2}(\omega)$ for all integers $j$.

Let $X_{K}(\alpha)$ be the hyperbolic cone-manifold with underlying space $S^3$ and with singular set $K$ of cone-angle $\alpha$. The volume of $X_{J(k,2n)}(\alpha)$ is given as follows.

\begin{thm} \label{odd}
We have $$\emph{Vol}(X_{J(2m+1,2n)}(\alpha)) = \int_{\alpha}^{\pi} \log \left| \frac{S_m(z) - M^2 S_{m-1}(z)}{M^2 S_m(z) - S_{m-1}(z)} \right| d\omega.$$
Here $M = e^{i\frac{\omega}{2}}$ and $z$, with $\emph{Im}\Big( S_m(z) \overline{S_{m-1}(z)} \Big) \le 0$, is a zero of the Riley polynomial
$$\Phi_{J(2m+1,2n)}(M,z)= S_{n}(t_m)- d_m S_{n-1}(t_m),$$
where 
\begin{eqnarray*}
t_m &=& M^2+M^{-2}+2-z-(z-2)(z-M^2-M^{-2})S_m(z)S_{m-1}(z),\\
d_m &=& 1 - (z-M^2-M^{-2})S_{m}(z) \left( S_{m}(z) - S_{m-1}(z) \right).
\end{eqnarray*}
\end{thm}

\begin{thm} \label{even}
We have $$\emph{Vol}(X_{J(2m,2n)}(\alpha)) = \int_{\alpha}^{\pi} \log \left| \frac{(S_m(z) - S_{m-1}(z)) -  M^2(S_{m-1}(z) - S_{m-2}(z))}{M^2(S_m(z) - S_{m-1}(z)) - (S_{m-1}(z) - S_{m-2}(z))} \right| d\omega.$$
Here $M = e^{i\frac{\omega}{2}}$ and $z$, with $\emph{Im} \Big( \big( S_m(z) - S_{m-1}(z) \big) \overline{S_{m-1}(z) - S_{m-2}(z)} \Big) \le 0$,  is a zero of the Riley polynomial
$$\Phi_{J(2m,2n)}(M,z)= S_{n}(\bar{t}_m)- \bar{d}_m S_{n-1}(\bar{t}_m)=0,$$
where
\begin{eqnarray*}
\bar{t}_m &=& 2+(z-2)(z-M^2-M^{-2})S^2_{m-1}(z),\\
\bar{d}_m &=& 1+(z-M^2-M^{-2})S_{m-1}(z) \left( S_{m}(z) - S_{m-1}(z) \right).
\end{eqnarray*}
\end{thm}

\begin{remark}
For a fixed integer $m$, the Riley polynomial $P_n:=\Phi_{J(2m+1,2n)}(M,z)$ can be defined recursively by
$$P_n = \Big( M^2+M^{-2}+2-z-(z-2)(z-M^2-M^{-2})S_m(z)S_{m-1}(z) \Big) P_{n-1} - P_{n-2}$$
for all integers $n$, with initial conditions $P_0=1$ and $$P_1=1+(z-M^2-M^{-2})S_{m-1}(z) \big( S_m(z) - S_{m-1}(z)\big).$$

Similarly, the Riley polynomial $Q_n:=\Phi_{J(2m,2n)}(M,z)$ can be defined recursively by
$$Q_n = \Big( 2+(z-2)(z-M^2-M^{-2})S^2_{m-1}(z) \Big) Q_{n-1} - Q_{n-2}$$
for all integers $n$, with initial conditions $Q_0=1$ and $$Q_1=1 - (z-M^2-M^{-2})S_{m-1}(z) \left( S_{m-1}(z) - S_{m-2}(z) \right).$$
\end{remark}

With appropriate changes of variables, we obtain the volume formulae for cone-manifolds of two-bridge knots with Conway notation $C(2n,3)$ in \cite{HL} and for cone-manifolds of twist knots in \cite{HMP} by setting $m=1$ in Theorems \ref{odd} and \ref{even}, respectively.

The paper is organized as follows. In Section \ref{Riley} we study the Riley polynomial of a class of one-relator groups which includes two-bridge knot groups. In Section \ref{pf} we apply the result in Section \ref{Riley} to calculate the Riley polynomial of the double twist knot $J(k,2n)$ and then we prove Theorems \ref{odd} and \ref{even}. 

\subsection{Acknowlegement} This work was partially supported by a grant from the Simons Foundation (\#354595 to Anh Tran).

\section{Nonabelian representations} \label{Riley}

In this section we will study nonabelian $SL_2(\BC)$-representations of a class of one-relator groups which includes two-bridge knot groups. Like the case of two-bridge knot groups, the set of nonabelian $SL_2(\BC)$-representations of each group in this class is described by a single polynomial in two variables which we call the Riley polynomial of the group. We will present three different approaches to the Riley polynomial. 

Let $F_{a,b}=\la a,b \ra$ be the free group on two letters $a$ and $b$. For a word $u \in F_{a,b}$ let $\tilde{u}$ be the word obtained from $u$ by replacing $a$ and $b$ by $b^{-1}$ and $a^{-1}$ respectively.

We consider the group 
\begin{equation} \label{grp}
G=\la a,b \mid wa = b w \ra,
\end{equation} 
where $w$ is a word in $F_{a,b}$ with $w \not= 1$ and $\tilde{w}=w^{-1}$. 

The knot group of a two-bridge knot always has a presentation of the form \eqref{grp}. Indeed, two-bridge knots are those knots admitting a projection with only two maxima and two minima. The double branched cover of $S^3$ along a two-bridge knot is a lens space $L(p,q)$, which is obtained by doing a $p/q$ surgery on the unknot. Such a two-bridge knot is denoted by $\fb(p,q)$. Here $p$ and $q$ are relatively prime odd integers, and we can always assume that $p > |q| \ge 1$. It is known that $\fb(p',q')$ is ambient isotopic to $\fb(p,q)$ if and only if $p'=p$ and $q' \equiv q^{\pm 1} \pmod{p}$, see e.g. \cite{BZ}. The knot group of the two-bridge knot $\fb(p,q)$ has a presentation of the form $\la a, b \mid wa = b w \ra$ where $a,b$ are meridians, $w=a^{\ve_1} b^{\ve_2} \cdots a^{\ve_{p-2}}b^{\ve_{p-1}}$ and $\ve_i=(-1)^{\lfloor iq/p \rfloor}$ for $1 \le i \le p-1$. Since $\ve_i = \ve_{p-i}$, we have $\tilde{w}=b^{-\ve_1} a^{-\ve_2} \cdots b^{-\ve_{p-2}}a^{-\ve_{p-1}}=b^{-\ve_{p-1}} a^{-\ve_{p-2}} \cdots b^{-\ve_{2}}a^{-\ve_{1}}=w^{-1}$.

We now consider representations of $G$ into $SL_2(\BC)$. Two representations $\rho, \rho': G \to SL_2(\BC)$ are called conjugate if there exists a matrix $C \in SL_2(\BC)$ such that $\rho'(g) = C \rho(g) C^{-1}$ for all $g \in G$. In this paper we study nonabelian representations up to conjugation. A representation $\rho: G \to SL_2(\BC)$ is called nonabelian if the image $\rho(G)$ is a nonabelian subgroup of $SL_2(\BC)$.  Suppose $\rho$ is nonabelian. Since $a$ and $b=waw^{-1}$ are conjugate, up to conjugation we can assume that 
\begin{equation} \label{conj}
\rho(a)= \left[ \begin{array}{cc}
M & 1 \\
0 & M^{-1} \end{array} \right] \quad \text{and} \quad \rho(b) = \left[ \begin{array}{cc}
M & 0 \\
r & M^{-1} \end{array} \right],
\end{equation}
where $(M,r) \in \BC^* \times \BC^*$ satisfies the matrix equation $\rho(w)\rho(a)=\rho(b)\rho(w)$. 

We will show that this matrix equation is equivalent to a single equation in $M$ and $y$, which we call the Riley polynomial of the group $G$.

For a word $u$ in $F_{a,b}$, we write $\rho(u)= \left[ \begin{array}{cc}
u_{11} & u_{12} \\
u_{21} & u_{22} \end{array} \right]$ where $u_{ij} \in \BC$.

\subsection{The Riley polynomial} To solve the matrix equation $\rho(w)\rho(a)=\rho(b)\rho(w)$, we follow Riley's approach in \cite{Ri}.

\begin{lemma} \label{2112}
Suppose $u$ is a word in $F_{a,b}$ with $u \not= 1$ and $\tilde{u}=u^{-1}$. Then we have
\begin{equation} \label{ur}
u_{21} = r u_{12}.
\end{equation}
\end{lemma}

\begin{proof}
We use induction on the length $\ell(u)$ of $u$. Note that $\ell(u) \ge 2$. If $\ell(u)=2$ then it is easy to see that $u=g_1^{\ve}g_2^{\ve}$, where $\{g_1,g_2\}=\{a,b\}$ and $\ve = \pm 1$. Equality \eqref{ur} holds true by a direct calculation.

Suppose $\ell(u) > 2$. Write $u=gvh$, where $g,h \in \{a^{\pm 1},b^{\pm 1}\}$ and $v \in F_{a,b}$ with $\ell(v)=\ell(r)-2$. Since $\tilde{u} = u^{-1}=h^{-1}v^{-1}g^{-1}$, we have $h^{-1}=\tilde{g}$, $g^{-1} = \tilde{h}$ and $v^{-1} = \tilde{v}$. It follows that $\{g,h\}=\{a,b\}$ or $\{g,h\}=\{a^{-1},b^{-1}\}$.

We consider the case $(g,h)=(a,b)$. Then $u=avb$. By induction hypothesis we have $\rho(v)=\left[ \begin{array}{cc}
v_{11} & v_{12} \\
v_{21} & v_{22} \end{array} \right]$ with $v_{21} = r v_{12}$. Hence 
\begin{eqnarray*}
\rho(u) &=& \rho(a) \rho(v) \rho(b) \\
&=& \left[ \begin{array}{cc}
M & 1 \\
0 & M^{-1} \end{array} \right] \left[ \begin{array}{cc}
v_{11} & v_{12} \\
v_{21} & v_{22} \end{array} \right] \left[ \begin{array}{cc}
M & 0 \\
r & M^{-1} \end{array} \right]\\
&=& \left[ \begin{array}{cc}
M^2 v_{11}+ M r v_{12} + M v_{21} +r v_{22} & u_{12} + M^{-1} v_{22} \\
v_{21} + M^{-1}r v_{22} & M^{-2} v_{22} \end{array} \right].
\end{eqnarray*}
It follows that $u_{21} = v_{21}+r M^{-1} v_{22} = r (v_{12} + M^{-1} v_{22}) = r u_{12}$.

The cases $(g,h)=(b,a)$, $(a^{-1},b^{-1})$ and $(b^{-1},a^{-1})$ can be proved similarly.
\end{proof} 

By Lemma \ref{2112} we have $\rho(w) = \left[ \begin{array}{cc}
w_{11} & w_{12} \\
w_{12} & w_{22} \end{array} \right]$ with $w_{21} = r w_{12}$. Then 
\begin{eqnarray*}
\rho(w)\rho(a) - \rho(b)\rho(w) &=& \left[ \begin{array}{cc}
w_{11} & w_{12} \\
w_{21} & w_{22} \end{array} \right] \left[ \begin{array}{cc}
M & 1 \\
0 & M^{-1} \end{array} \right] - \left[ \begin{array}{cc}
M & 0 \\
r & M^{-1} \end{array} \right] \left[ \begin{array}{cc}
w_{11} & w_{12} \\
w_{21} & w_{22} \end{array} \right] \\
&=& \left[ \begin{array}{cc}
0 & w_{11} - (M-M^{-1})w_{12} \\
-r w_{11} + (M-M^{-1})w_{21} & -r w_{12} + w_{21} \end{array} \right] \\
&=& \left[ \begin{array}{cc}
0 & w_{11} - (M-M^{-1})w_{12} \\
-r \big( w_{11} - (M-M^{-1})w_{12} \big) & 0 \end{array} \right].
\end{eqnarray*}
Hence $\rho(w)\rho(a) = \rho(b)\rho(w)$ if and only if 
$$w_{11} - (M-M^{-1})w_{12} = 0.$$
We call $w_{11} - (M-M^{-1})w_{12}$ the Riley polynomial of the group $G$. It describes the set of nonabelian representations of $G$ into $SL_2(\BC)$.

\subsection{L\^e's approach} In this subsection we determine the universal $SL_2(\BC)$-character ring of the group $G$ defined in \eqref{grp}. As a consequence, we obtain another description of the set of nonabelian $SL_2(\BC)$-representations of $G$. We will follow L\^e's approach in \cite{Le}.

We first recall the definitions of the character variety and universal character ring of a group. The set of representations of a finitely generated group $H$ into $SL_2(\BC)$ is an algebraic set defined over $\BC$, on which
$SL_2(\BC)$ acts by conjugation. The set-theoretic
quotient of the representation space by that action does not
have good topological properties, because two representations with
the same character may belong to different orbits of that action. A better
quotient, the algebro-geometric quotient denoted by $\chi(H)$
(see \cite{CS, LM}), has the structure of an algebraic
set. There is a bijection between $\chi(H)$ and the set of all
characters of representations of $H$ into $SL_2(\BC)$, hence
$\chi(H)$ is usually called the character variety of $H$. 

The character variety of $H$ is determined by the traces of some fixed elements $h_1, \cdots, h_k$ in $H$. More precisely, we can find $h_1, \cdots, h_k$ in $H$ such that for every element $h$ in $H$ there exists a polynomial $P_h$ in $k$ variables such that for any representation $\rho: H \to SL_2(\BC)$ we have $\tr \rho(h) = P_h(x_1, \cdots, x_k)$ where $x_i:=\tr \rho(h_i)$. The universal character ring of $H$ is defined to be the quotient of the polynomial ring $\BC[x_1, \cdots, x_k]$ by the ideal generated by all expressions of the form $\tr\rho(u)-\tr\rho(v)$, where $u$ and $v$ are any two words in the letters $g_1, \cdots, g_k$ which are equal in $H$, see e.g. \cite{BH}. The universal character ring of $H$ is actually independent of the choice of $h_1, \cdots, h_k$. The quotient of the universal character ring of $H$ by its nilradical is equal to the character ring of $H$, which is the ring of regular functions on the character variety $X(H)$.

Recall that $F_{a,b}$ is the free group on two letters $a$ and $b$. The character variety of $F_{a,b}$ is isomorphic to $\BC^3$ by the Fricke-Klein-Vogt theorem, see \cite{Fr, Vo}. It follows that for every word $u \in F_{a,b}$ there is a unique polynomial $P_u$ in 3 variables such that for any representation $\rho: F_{a,b} \to SL_2(\BC)$ we have $\tr \rho(u)=P_u (x,x',y)$ where $x:=\tr\rho(a),~x':=\tr\rho(b)$ and $y:=\tr\rho(ab)$.  

\begin{lemma} \label{tilder}
For $u \in F_{a,b}$ we have that $x-x'$ divides $\tr \rho(u) - \tr \rho(\tilde{u})$ in $\BC[x,x',y]$.
\end{lemma}

\begin{proof}
Since $\tilde{u}$ is the word obtained from $u$ by replacing $a$ by $b^{-1}$ and $b$ by $a^{-1}$, we have
$$
\tr \rho(\tilde{u}) = P_u(\tr \rho(b^{-1}), \tr \rho(a^{-1}), \tr \rho(b^{-1}a^{-1})) = P_u(x',x,y).
$$
Hence $\tr \rho(u) - \tr \rho(\tilde{u}) = P_u(x,x',y) - P_u(x',x,y)$ is divisible by $x-x'$ in $\BC[x,x',y]$.
\end{proof}

\begin{proposition} \cite[Prop.2.1]{Tran-universal}
\label{univ}
Let $H=\la a,b \mid u =v \ra$, where $u, v \in F_{a,b}$. Then the universal character ring of $H$ is the quotient of the polynomial ring $\BC[x,y,z]$ by the ideal generated by the four polynomials $\tr \rho(u) - \tr \rho(v)$, $\tr \rho(u a^{-1}) -\ tr \rho(v a^{-1})$, $\tr \rho(ub^{-1}) - \tr \rho(vb^{-1})$ and $\tr \rho(u a^{-1}b^{-1}) - \tr \rho(v a^{-1} b^{-1})$.
\end{proposition}

Recall that $G=\la a, b \mid wa = bw\ra$ where $\tilde{w} = w^{-1}$. We now describe the universal character ring of $G$. In this case, since $a$ and $b=waw^{-1}$ are conjugate, we have $\tr \rho(a) = \tr \rho(b)$. This means that $x=x'$.

\begin{thm}
The universal character ring of $G$ is the quotient of the polynomial ring $\BC[x,y]$ by the principal ideal generated by the polynomial $\tr \rho(bwa^{-1})-\tr \rho(w).$
\label{main}
\end{thm}

\begin{proof}
By Proposition \ref{univ} the universal character ring of $G$ is the quotient of the polynomial ring $\BC[x,z]$ by the ideal generated by the four polynomials 
\begin{eqnarray*}
&&\tr \rho(wa) - \tr \rho(bw), \\
&& \tr \rho((wa) a^{-1}) - \tr \rho((bw) a^{-1}) = \tr \rho(w) - \tr \rho(bwa^{-1}),\\
&& \tr \rho((wa)b^{-1}) - \tr \rho((bw)b^{-1}) = \tr \rho(wab^{-1}) - \tr \rho(w),\\
&&\tr \rho((wa) a^{-1}b^{-1}) - \tr \rho((bw) a^{-1} b^{-1}) = \tr \rho(wb^{-1}) - \tr \rho(wa^{-1}).
\end{eqnarray*}

Since $x=x'$, by Lemma \ref{tilder} we have $\tr \rho(u) = \tr \rho(\tilde{u})$ for all $u \in F_{a,b}$. Hence
\begin{eqnarray*}
\tr \rho(wa) - \tr \rho(bw) &=& \tr \rho(\widetilde{wa}) - \tr \rho(bw) = \tr \rho(\tilde{w} \tilde{a}) - \tr \rho(bw) \\
&=& \tr \rho(w^{-1}b^{-1}) - \tr \rho(bw) = 0,\\
\tr \rho(wab^{-1}) - \tr \rho(w) &=& \tr \rho(\widetilde{wab^{-1}}) - \tr \rho(w) = \tr \rho(\tilde{w}b^{-1}a) - \tr \rho(w) \\
&=& \tr \rho(w^{-1}b^{-1}a) - \tr \rho(w) = \tr \rho(bwa^{-1}) - \tr \rho(w),\\
\tr \rho(wb^{-1}) - \tr \rho(wa^{-1}) &=& \tr \rho(\widetilde{wb^{-1}}) - \tr \rho(wa^{-1}) = \tr \rho(\tilde{w}a) - \tr \rho(wa^{-1}) \\
&=& \tr \rho(w^{-1}a) - \tr \rho(wa^{-1}) = 0.
\end{eqnarray*}
The theorem then follows.
\end{proof}

Now suppose $\rho: G \to SL_2(\BC)$ is a nonabelian representation of the form \eqref{conj}. Then
\begin{eqnarray*}
\tr \rho(bwa^{-1}) - \tr \rho(w) &=& (w_{11}+w_{22} -rw_{11} - M^{-1} w_{21}+r M w_{12}) - (w_{11} + w_{22}) \\
&=& -r (w_{11} - M w_{12} + M^{-1}r^{-1} w_{21}).
\end{eqnarray*}
Hence, by Theorem \ref{main}, we have $\rho(wa) = \rho(bw)$ if and only if $$w_{11} - (M w_{12} - M^{-1}r^{-1} w_{21})=0.$$
We call $w_{11} - (M w_{12} - M^{-1}r^{-1} w_{21})$ the L\^e polynomial of $G$.

\begin{remark}
Since $w_{21} = r w_{12}$ (by Lemma \ref{2112}), it is easy to see that the L\^e equation is equal to the Riley equation.
\end{remark}

\subsection{Mednykh's approach} In this subsection we study the set of nonabelian $SL_2(\BC)$-representations of the group $G$ defined in \eqref{grp}, following Mednykh's approach in \cite{HMP}. 

\begin{lemma} \label{-I}
Suppose $C \in SL_2(\BC)$. Then $C^2 = -I$ if and only if $\tr C=0$.
\end{lemma}

\begin{proof}
By the Cayley-Hamilton theorem we have $C^2-(\tr C)C+I=0$. It follows that $C^2+I=0$ if and only if $\tr C=0$.
\end{proof}

\begin{proposition} \label{CC}
Suppose $C \in SL_2(\BC)$ such that $C \rho(b) = \rho(a^{-1})C$ and $C^2=-I$. Then for any word $u \in F_{a,b}$ we have
\begin{equation} \label{C}
C \rho(u) =\rho(\tilde{u}) C.
\end{equation}
\end{proposition}

\begin{proof} By taking the inverse of $C \rho(b) = \rho(a^{-1})C$ we get $C \rho(b^{-1}) = \rho(a)C$.

Since $C=-C^{-1}$ we have
$$C \rho(a) = -C^{-1}\rho(a) = -(\rho(a^{-1})C)^{-1} = - (C \rho(b))^{-1} = - \rho(b)^{-1} C^{-1} = \rho(b^{-1}) C.$$
By taking the inverse of $C \rho(a) = \rho(b^{-1})C$ we get $C \rho(a^{-1}) = \rho(b)C$. Hence $$
C \rho(u) =\rho(\tilde{u}) C.
$$ for $u \in \{a^{\pm 1}, b^{\pm 1}\}$.

We now prove \eqref{C} by induction on the length $\ell(u)$ of $u \in F_{a,b}$. If $\ell(u)=1$, then $u \in  \{a^{\pm 1}, b^{\pm 1}\}$. By the above arguments we have $
C \rho(u) =\rho(\tilde{u}) C.
$

Suppose $\ell(u)>1$. Then we can write $u=g v$ where $g \in \{a^{\pm 1}, b^{\pm 1}\}$ and $v \in F_{a,b}$ with $\ell(v) = \ell(u)-1$. By induction hypothesis $C \rho(v) =\rho(\tilde{v}) C$. We have
$$\rho(\tilde{u}) C = \rho(\tilde{g}) \rho(\tilde v) C = \rho(\tilde{g}) C \rho(v) = C \rho(g) \rho(v)  = C \rho(u) .$$
The proposition follows.
\end{proof}

Now consider $G = \la a,b \mid wa=bw \ra$ where $w \in F_{a,b}$ with $\tilde{w} = w^{-1}$. Suppose $\rho: F_{a,b} \to SL_2(\BC)$ is a nonabelian representation of the form \eqref{conj}. 

Let $C = \left[ \begin{array}{cc}
0 & -1/\sqrt{r} \\
\sqrt{r} & 0 \end{array} \right]$. Then it is easy to check that $C \rho(b)= \rho(a^{-1})C$ and $C^2=-I$. By Proposition \ref{CC} we have $C \rho(u) =\rho(\tilde{u}) C$ for all $u \in F_{a,b}$. 

Since $w^{-1} = \tilde{w}$ we have
\begin{eqnarray*}
\rho(bwa^{-1}w^{-1}) = - \rho(bwa^{-1}\tilde{w}) C^2 = -\rho(bwa^{-1}) C \rho(w) C = -\rho(bw) C \rho(bw) C.
\end{eqnarray*}
Hence $\rho(bwa^{-1}w^{-1}) = I$ if and only if $\big( \rho(bw) C \big)^2 = -I$. By Lemma \ref{-I}, this is equivalent to $\tr\big( \rho(bw) C \big) = 0$. Since $\tr\big( \rho(bw) C \big) = \sqrt{r} \, M w_{12} - (r w_{11} + M^{-1} w_{21})/\sqrt{r},$
we conclude that $\rho(wa) = \rho(bw)$ if and only if $$w_{11}-(Mw_{12}-M^{-1}r^{-1}w_{21})=0.$$
We call $w_{11}-(Mw_{12}-M^{-1}r^{-1}w_{21})$ the Mednykh polynomial of $G$. Note that it is exactly the L\^e equation, which is equal to the Riley polynomial (by Lemma \ref{2112}).

\section{Proofs of Theorems \ref{odd} and \ref{even}}

\label{pf}

In this section we apply the result in Section \ref{Riley} to calculate the Riley polynomial of the double twist knot $J(k,2n)$ and then we prove Theorems \ref{odd} and \ref{even}. 

\subsection{The Riley polynomial} By \cite{HS} the knot group of $K=J(k,2n)$ has a presentation
$\pi_1(X_K) = \la a,b~|~w^na=bw^n \ra,$ where $a,b$ are meridians and
$$w = 
\begin{cases} 
(ba^{-1})^mba(b^{-1}a)^m, & \text{if~} k=2m+1,\\
(ba^{-1})^m(b^{-1}a)^m, & \text{if~} k=2m.
\end{cases}$$

Suppose $\rho: \pi_1(X_K) \to SL_2(\BC)$ is a nonabelian representation. 
Taking conjugation if necessary, we can assume that $\rho$ 
has the form 
\begin{equation}
\rho(a)=\left[ \begin{array}{cc}
M & 1\\
0 & M^{-1} \end{array} \right] \quad \text{and} \quad \rho(b)=\left[ \begin{array}{cc}
M & 0\\
2-z & M^{-1} \end{array} \right]
\label{nonabelian}
\end{equation}
where $(M,z) \in \BC^* \times \BC$ satisfies the matrix equation $\rho(w^na)=\rho(bw^n)$. Note that $z =\tr \rho(ab^{-1})=M^2+M^{-2}+2 - \tr \rho(ab)$.

The following lemmas are elementary, see e.g. \cite{Tran-tap}.

\begin{lemma} \label{chev} We have $$S^2_j(\omega) - \omega S_j(\omega) S_{j-1}(\omega) + S^2_{j-1}(\omega)=1.$$
\end{lemma}

\begin{lemma}
\label{formula}
Suppose $V = \left[ \begin{array}{cc}
v_{11} & v_{12} \\
v_{21} & v_{22} \end{array} \right] \in SL_2(\BC)$. Then 
$$
V^j = \left[ \begin{array}{cc}
S_{j}(v) - v_{22} S_{j-1}(v) & v_{12} S_{j-1}(v) \\
v_{21} S_{j-1}(v) & S_{j}(v) - v_{11} S_{j-1}(v) \end{array} \right],
$$
where $v:= \tr V = v_{11}+v_{22}$.
\end{lemma}

\subsubsection{The case $k=2m+1$} In this case we have $w=(ba^{-1})^m ba (b^{-1}a)^m$.

\begin{proposition} \label{w}
We have $\rho(w)=\left[ \begin{array}{cc}
w_{11} & w_{12} \\
(2-z)w_{12} & w_{22} \end{array} \right]$ where
\begin{eqnarray*}
w_{11} &=& M^2 S^2_{m}(z) - 2M^2 S_m(z)S_{m-1}(z) + (2+M^2-z)S^2_{m-1}(z),\\
w_{12} &=& \big( S_m(z) - S_{m-1}(z) \big) \big( M S_m(z) - M^{-1} S_{m-1}(z) \big),\\
w_{22} &=& (M^{-2} + 2 - z) S^2_{m}(z) - 2 M^{-2} S_m(z)S_{m-1}(z) + M^{-2} S^2_{m-1}(z).
\end{eqnarray*}
\end{proposition}

\begin{proof}
Since $\rho(ba^{-1})=\left[ \begin{array}{cc}
1 & -M \\
M^{-1}(2-z) & z-1 \end{array} \right]$, by Lemma \ref{formula} we have 
$$\rho((ba^{-1})^m)=\left[ \begin{array}{cc}
S_m(z)-(z-1)S_{m-1}(z) & -MS_{m-1}(z) \\
M^{-1}(2-z)S_{m-1}(z) & S_m(z)-S_{m-1}(z) \end{array} \right].$$
Similarly, $$\rho((b^{-1}a)^m)=\left[ \begin{array}{cc}
S_m(z)-(z-1)S_{m-1}(z) & M^{-1}S_{m-1}(z) \\
M(z-2)S_{m-1}(z) & S_m(z)-S_{m-1}(z) \end{array} \right].$$
Since $\rho(w)=\rho((ba^{-1})^m ba (b^{-1}a)^m)$, the lemma follows by a direct calculation.
\end{proof}

\begin{proposition} \label{R}
The Riley polynomial of $J(2m+1,2n)$ is 
$$\Phi_{J(2m+1,2n)}(M,z) = S_n(t_m) - \Big( 1  - (z-M^2-M^{-2})S_{m}(z) \big( S_m(z)- S_{m-1}(z) \big) \Big) S_{n-1}(t_m)$$
where $t_m:= \tr \rho(w) = (M^2+M^{-2}+2-z) - (z-2)(z-M^2-M^{-2}) S_m(z) S_{m-1}(z)$.
\end{proposition}

\begin{proof}
By Proposition \ref{w} we have 
$$t_m = w_{11} + w_{22} = (M^2+M^{-2}+2-z)(S^2_{m}(z) + S^2_{m-1}(z))  - 2(M^2+M^{-2})S_m(z)S_{m-1}(z).$$
Since $S^2_{m}(z)+S^2_{m-1}(z)=1+zS_m(z)S_{m-1}(z)$ (by Lemma \ref{chev}), we have 
\begin{eqnarray*}
t_m 
&=& (M^2+M^{-2}+2-z)(1 + z S_{m}(z)S_{m-1}(z))  - 2(M^2+M^{-2})S_m(z)S_{m-1}(z)\\
&=& (M^2+M^{-2}+2-z) - (z-2)(z-M^2-M^{-2}) S_m(z) S_{m-1}(z).
\end{eqnarray*}

Since $\rho(w)=\left[ \begin{array}{cc}
w_{11} & w_{12} \\
(2-z)w_{12} & w_{22} \end{array} \right]$, by Lemma \ref{formula} we have $$\rho(w^n) = \left[ \begin{array}{cc}
S_n(t_m)-w_{22} \, S_{n-1}(t_m) & w_{12} \, S_{n-1}(t_m) \\
(2-z)w_{12} \, S_{n-1}(t_m) & S_n(t_m)-w_{11} \, S_{n-1}(t_m) \end{array} \right].$$ Hence the Riley polynomial is 
\begin{eqnarray*}
\Phi_K(M,z) &=&  S_n(t_m)-w_{22} \, S_{n-1}(t_m) - (M-M^{-1})w_{12} \, S_{n-1}(t_m) \\
&=& S_n(t_m) - \big( w_{22} - (M-M^{-1})w_{12} \big)  S_{n-1}(t_m).
\end{eqnarray*}

Since $S^2_{m}(z)+S^2_{m-1}(z)-zS_m(z)S_{m-1}(z)=1$ we have
\begin{eqnarray*}
&&w_{22}+(M-M^{-1})w_{12}\\
 &=& (1+M^2+M^{-2}-z) S^2_{m}(z) - (M^2+M^{-2})S_m(z)S_{m-1}(z) + S^2_{m-1}(z)\\
&=&1  - (z-M^2-M^{-2})S_{m}(z) \big( S_m(z)- S_{m-1}(z) \big).
\end{eqnarray*}
The formula for $\Phi_K(M,z)$ then follows.
\end{proof}

\subsubsection{The case $k=2m$} In this case we have $w = (ba^{-1})^m(b^{-1}a)^m$.

\begin{proposition} \label{wbar}
We have $\rho(w)=\left[ \begin{array}{cc}
w_{11} & w_{12} \\
(2-z)w_{12} & w_{22} \end{array} \right]$ where
\begin{eqnarray*}
w_{11} &=& S^2_{m}(z) + (2-2z)S_m(z)S_{m-1}(z) + (1+2M^2-2z-M^2z+z^2)S^2_{m-1}(z),\\
w_{12} &=& (M^{-1}-M)S_m(z)S_{m-1}(z) + (M^{-1}+M-M^{-1}z)S^2_{m-1}(z),\\
w_{22} &=& S^2_{m}(z) - 2S_m(z)S_{m-1}(z) + (1+2M^{-2}-M^{-2}z)S^2_{m-1}(z).
\end{eqnarray*}
\end{proposition}

\begin{proposition} \label{Rbar}
The Riley polynomial of $J(2m,2n)$ is 
$$\Phi_{J(2m,2n)}(M,z) = S_n(\bar{t}_m) - \Big( 1  + (z-M^2-M^{-2})S_{m-1}(z) \big( S_m(z)- S_{m-1}(z) \big) \Big) S_{n-1}(\bar{t}_m)$$
where $\bar{t}_m:= \tr \rho(w) = 2 + (z-2)(z-M^2-M^{-2})S^2_{m-1}(z)$.
\end{proposition}

\begin{remark}
Similar formulae for the Riley polynomial of $J(2m+1,2n)$ and $J(2m,2n)$ have already been obtained in \cite{MPL, MT}.
\end{remark}

\subsection{The canonical longitude} Recall that $X_K$ is the complement of the knot $K$. The
boundary of $X_K$ is a torus $\BT^2$. There is a standard choice of a meridian $\mu$ and a longitude
$\lambda$ on $\BT^2$ such that the linking number between the longitude and the knot is zero. We call
$\lambda$ the canonical longitude of $K$ corresponding to the meridian $\mu$. 

Let $\mu=a$ be the meridian of $K = J(k, 2n)$ and $\lambda$ the canonical longitude corresponding
to $\mu$.  Then we have $$\lambda=\begin{cases} 
\overleftarrow{w}^n w^n \mu^{-4n}, & \text{if~} k=2m+1,\\
\overleftarrow{w}^n w^n, & \text{if~} k=2m,
\end{cases}$$
where $\overleftarrow{w}$ is the word in the letters $a,b$ obtained by writing $w$ in the reversed order. With the representation $\rho$  in \eqref{nonabelian}, by \cite{HS} we have 
$\rho(\lambda) =\left[ \begin{array}{cc}
L & * \\
0 & L^{-1} \end{array} \right]$, where $L=-\widetilde{w}_{12}/w_{12}$. Here $\widetilde{w}_{ij}$ is obtained from $w_{ij}$ by replacing $M$ by $M^{-1}$. 

From Propositions \ref{w} and \ref{wbar} we have the following.

\begin{proposition} \label{L}
(i) If $k=2m+1$ then 
$$L = - M^{-4n} \, \frac{M^{-1} S_m(z) - M S_{m-1}(z) }{ M S_m(z) - M^{-1}S_{m-1}(z)  }.$$

(ii) If $k=2m$ then 
$$L = - \frac{M^{-1}(S_m(z) - S_{m-1}(z)) - M(S_{m-1}(z) - S_{m-2}(z))}{M(S_m(z) - S_{m-1}(z)) - M^{-1} (S_{m-1}(z) - S_{m-2}(z))}.$$
\end{proposition}

\subsection{Proof of Theorems \ref{odd} and \ref{even}} We begin with a simple lemma.

\begin{lemma} \label{LL}
Suppose $z_1, z_2 \in \BC$ and $\omega \in \BR$. Then 
$$|z_1 - e^{i\omega} z_2|^2 = |z_1|^2+|z_2|^2-2 \, \emph{Re}( z_1 \overline{z_2}) \cos\omega - 2 \,\emph{Im}( z_1 \overline{z_2}) \sin\omega.$$
Hence $|z_1 - e^{i\omega} z_2| \ge |z_1 - e^{-i\omega} z_2|$ if and only if $\emph{Im}( z_1 \overline{z_2}) \sin\omega \le 0$. Moreover,  $|z_1 - e^{i\omega} z_2| = |z_1 - e^{-i\omega} z_2|$ if and only if $\emph{Im}( z_1 \overline{z_2}) \sin\omega = 0$.
\end{lemma}

For a hyperbolic two-bridge knot $K$, by \cite{HLM, Ko, Po, PW} there exists an angle $\alpha_K \in [\frac{2\pi}{3},\pi)$ such that $X_K(\alpha)$ is hyperbolic for $\alpha \in (0, \alpha_K)$, Euclidean for $\alpha = \alpha_K$, and spherical for $\alpha \in (\alpha_K, \pi]$. 

We first consider the case $K=J(2m+1,2n)$. For $\alpha \in (0, \alpha_K)$, by the Schlafli formula the volume of a hyperbolic cone-manifold $X_K(\alpha)$ is given by
$$\text{Vol}(X_K(\alpha)) = \int_{\alpha}^{\alpha_K} \frac{\ell_\omega}{2}d\omega,$$
where $\ell_{\omega}$ is the real length of the longitude of the cone-manifold $X_K(\omega)$. See e.g. \cite{HMP}.

For each $\alpha \le \omega \le \alpha_K$, $\ell_\omega$ is calculated as follows. Suppose $\rho: \pi_1(X_K) \to SL_2(\BC)$ is a nonabelian representation of the form \eqref{nonabelian}, where $M=e^{i\omega/2}$ and $z$ is a zero of the Riley polynomial $\Phi_K(M,z)$. The complex length of a canonical longitude $\lambda$ of $K$ is the complex number $\gamma_{\omega}$ module $2\pi\BZ$ satisfying
$$\tr \rho(\lambda) = 2\cosh \frac{\gamma_\omega}{2}.$$
Then $\ell_{\omega} = |\text{Re}(\gamma_{\omega})|$. By Proposition \ref{L} we have $\rho(\lambda) = \left[ \begin{array}{cc}
L & * \\
0 & L^{-1} \end{array} \right]$ where $$L = - M^{-4n} \, \frac{M^{-1} S_m(z) - M S_{m-1}(z) }{ M S_m(z) - M^{-1}S_{m-1}(z)}.$$ For the volume, we can either choose $|L| \ge 1$ or $|L| \le 1$. We choose $L$ with $|L| \ge 1$. Since $M=e^{i\frac{\omega}{2}}$, by Lemma \ref{LL} we have $\text{Im}( S_m(z) \overline{S_{m-1}(z)} ) \le 0$. Moreover $|L|=1$ if and only if $\text{Im}( S_m(z) \overline{S_{m-1}(z)} ) = 0$.

Since $\ell_\omega = |\text{Re}(\gamma_{\omega})| = 2 \log |L|$ we have 
$$
\text{Vol}(X_{J(2m+1,2n)}(\alpha)) = \int_{\alpha}^{\alpha_K} \frac{\ell_\omega}{2}d\omega 
= \int_{\alpha}^{\alpha_K} \log |L| d\omega.
$$
For $\alpha_K < \alpha \le \pi$, by \cite[Prop.6.4]{PW} all the characters are real. In particular $z=\tr \rho(ab^{-1}) \in \BR$, and hence $|L|=1$, for $\alpha_K < \alpha \le \pi$. Hence 
$$
\text{Vol}(X_{J(2m+1,2n)}(\alpha)) = 
 \int_{\alpha}^{\pi} \log |L| d\omega = \int_{\alpha}^{\pi} \log \left| \frac{S_m(z) - M^2 S_{m-1}(z) }{ M^2 S_m(z) - S_{m-1}(z)  } \right|  d\omega.
$$
This completes the proof of Theorem \ref{odd}.  The proof of Theorem \ref{even} is similar. In that case we apply Propositions \ref{wbar}, \ref{Rbar} and the fact that $$|L|=\left| \frac{(S_m(z) - S_{m-1}(z)) - M^{2} (S_{m-1}(z) - S_{m-2}(z))}{M^{2}(S_m(z) - S_{m-1}(z)) - (S_{m-1}(z) - S_{m-2}(z))} \right| \ge 1$$
if and only if $\text{Im} \Big( \big(S_m(z) - S_{m-1}(z)  \big) \overline{S_{m-1}(z) - S_{m-2}(z)} \Big) \le 0$.


\begin{thebibliography}{99999}

\bibitem[BH]{BH} G. Brumfiel and H. Hilden, {\em Sl(2) representations of finitely presented groups}, Contemp. Math. \textbf{187} (1995).

\bibitem[BZ]{BZ} G. Burde and H. Zieschang, {\em Knots}, de Gruyter Stud. Math., vol. 5, de Gruyter, Berlin, 2003.

\bibitem[CS]{CS} M. Culler and P. Shalen, {\em Varieties of group representations and
  splittings of 3-manifolds}, Ann. of Math. (2) \textbf{117} (1983), no. 1,
  109--146.
  
\bibitem[Fr]{Fr} R. Fricke, {\em Uber die theorie der automorphen modulgruppen}, Kgl. Ges. d. W. Nachrichten, Math-Phys. Klasse, (1896), 91--101.

\bibitem[HL]{HL} J.-Y. Ham and J. Lee, {\em The volume of hyperbolic cone-manifolds of the knot with Conway's notation $C(2n,3)$}, preprint 2015, arXiv:1512.05481.

\bibitem[HLM]{HLM} H. Hilden, M. Lozano, and J. Montesinos-Amilibia, {\em Volumes and Chern-Simons invariants of cyclic coverings over rational knots},  in Topology and Teichm\"uller spaces (Katinkulta, 1995), pages 31--55. World Sci. Publ., River Edge, NJ, 1996.

\bibitem[HMP]{HMP} J.-Y. Ham, A. Mednykh, and V. Petrov, {\em identities and volumes of the hyperbolic twist knot cone-manifolds}, J. Knot Theory Ramifications \textbf 23(2014) 1450064.

\bibitem[HS]{HS} J. Hoste and P. Shanahan, {\em A formula for the A-polynomial of twist knots}, J. Knot Theory Ramifications \textbf{13} (2004), no. 2, 193--209.

\bibitem[Ko1]{Ko} S. Kojima, {\em Deformations of hyperbolic 3-cone-manifolds}, J. Differential Geom. \textbf{49} (1998) 469--516.

\bibitem[Ko2]{Ko2} S. Kojima, {\em Hyperbolic 3-manifolds singular along knots}, Chaos Solitons Fractals \textbf{9} (1998) 765--777.

\bibitem[Le]{Le} T. Le, \textit{Varieties of representations and their subvarieties of cohomology jumps for certain knot groups}, Russian Acad. Sci. Sb. Math. \textbf{78} (1994) 187--209.

\bibitem[LM]{LM} A. Lubotzky and A. Magid, {\em Varieties of representations of finitely
generated groups}, Memoirs of the AMS {\bf 336} (1985).
  
\bibitem[MPL]{MPL} M. Macasieb, K. Petersen and R. van Luijk, {\em On character varieties of two-bridge knot groups} Proc. Lond. Math. Soc. (3) \textbf{103} (2011), no. 3, 473--507.

\bibitem[Me]{Me} A. Mednykh, {\em The volumes of cone-manifolds and polyhedra} http://mathlab.snu.ac.kr/\~top/workshop01.pdf, 2007. Lecture Notes, Seoul National University.

\bibitem[MR]{MR} A. Mednykh and A. Rasskazov, {\em Volumes and degeneration of cone-structures on the figure-eight knot}, Tokyo J. Math. \textbf{29} (2006) 445--464.

\bibitem[MT]{MT} T. Morifuji and A. Tran, {\em Twisted Alexander polynomials of two-bridge knots for parabolic representations}, Pacific J. Math. \textbf{269} (2014), no. 2, 433--451.

\bibitem[Po]{Po} J. Porti, {\em Spherical cone structures on 2-bridge knots and links}, Kobe J. Math. \textbf{21} (2004) 61--70.

\bibitem[PW]{PW} J. Porti and H. Weiss, {\em Deforming Euclidean cone 3-manifolds}, Geom. Topol. \textbf{11} (2007) 1507--1538.

\bibitem[Ri]{Ri} R. Riley, {\em Nonabelian representations of 2-bridge knot groups}, Quart. J. Math. Oxford Ser. (2) \textbf{35}
(1984), 191--208.

\bibitem[Th]{Th} W. Thurston, {\em The geometry and topology of 3-manifolds}, http://library.msri.org/books/gt3m, 1977/78. Lecture Notes, Princeton University.

\bibitem[Tr1]{Tran-universal} A. Tran, {\em The universal character ring of some families of one-relator groups}, Algebr. Geom.
Topol. \textbf{13} (2013), no. 4, 2317--2333.

\bibitem[Tr2]{Tran-tap} A. Tran, {\em Twisted Alexander polynomials of genus one two-bridge knots}, preprint 2015, arXiv:1506.05035.

\bibitem[Vo]{Vo} M. Vogt, {\em Sur les invariants fondamentaux des equations differentielles
lin\'eaires du second ordre}, Ann. Sci. \'Ecol. Norm. Sup\'er. Troi. \textbf{6} (1889) 3--71.

\end{thebibliography}
\end{document}